\documentclass{article}
\usepackage{graphicx} 
\usepackage{amsthm}
\usepackage{amsmath}
\usepackage{mathtools}
\usepackage{amssymb}
\usepackage{xcolor}
\usepackage{todonotes}
\usepackage{cite}
\usepackage[hidelinks]{hyperref}
\usepackage{cleveref}
\usepackage{authblk}
\usepackage{makecell}
\usepackage{booktabs}
\usepackage{multirow}
\usepackage{rotating}    
\usepackage{array}
\usepackage{footnote} 
\usepackage{tikz}
\usepackage{pgfplots}
\pgfplotsset{compat=1.18}

\newtheorem{theorem}{Theorem}

\newtheorem{proposition}[theorem]{Proposition}
\newtheorem{observation}[theorem]{Observation}

\newcommand{\nc}{\newcommand}
\nc{\beginquote}{“}
\nc{\bR}{\mathbb{R}}

\nc{\st}{dx} 
\nc{\stc}{dc}
\nc{\sts}{ds}
\nc{\stP}{dP}
\nc{\stQ}{dQ}
\nc{\opfcost}{F} 
\nc{\stS}{dS} 
\nc{\boY}{{\mathbf{Y}}}
\nc{\boI}{{\mathbf{I}}}
\nc{\boV}{{\mathbf{V}}}
\nc{\boS}{{\mathbf{S}}}
\nc{\boG}{\mathcal{G}}
\nc{\integers}{\mathbb{Z}}
\nc{\cN}{\mathcal{N}}
\nc{\cC}{\mathcal{C}}
\nc{\cI}{\mathcal{I}}
\nc{\cP}{\mathcal{P}}
\nc{\cE}{\mathcal{E}}
\nc{\cS}{\mathcal{S}}
    
\nc{\barra}{|}
\nc{\boundsym}{b}
\newcommand{\abs}[1]{\left|#1\right|}
\nc{\lowb}{\mathbf{\boundsym}^0}
\nc{\upb}{\mathbf\boundsym^1}
\newcommand{\bound}[2]{\boundsym^{#1}_{#2}}

\newcommand{\boundu}[1]{\bound{2}{v}}
\newcommand{\boundlv}[1]{\bound{l}{v}}
\newcommand{\bounduv}[1]{\bound{2}{v}}
\DeclareMathOperator{\monom}{\omega}
\nc{\hyperC}{\text{P}}
\nc{\ConvexHull}[1]{\text{conv}(#1)}
\nc{\ConcaveHull}[1]{\text{conc}(#1)}

\nc{\vertex}{v}
\nc{\wertex}{w}
\nc{\ivertex}{v}
\nc{\evertex}{w}
\nc{\arc}{e}
\nc{\aarc}{\overline{\arc}}
\nc{\circuit}{C}
\nc{\sizecircuit}{L}
\nc{\digraph}{D}
\nc{\cycle}{C}
\nc{\Path}{P}
\nc{\V}[1]{V_{#1}}
\nc{\Vk}{\V{k}}
\nc{\Vm}{\V{m}}
\nc{\buses}{\mathcal{B}}
\nc{\lines}{\mathcal{L}}
\nc{\bus}{k}
\nc{\buss}{m}
\nc{\lin}{\bus\buss}
\nc{\rlin}{\buss\bus}
\nc{\cedge}[1]{c_{#1}}
\nc{\ckm}{\cedge{\lin}}
\nc{\elecgraph}{\mathcal{N}}
\nc{\genset}{\mathcal{G}}
\nc{\pow}{P}
\nc{\rpow}{Q}
\newcommand{\injP}[1]{\pow_{#1}}
\newcommand{\injQ}[1]{\rpow_{#1}}
\nc{\pl}{\injP{\lin}}
\nc{\ql}{\injQ{\lin}}
\newcommand{\pg}[1]{P^G_{#1}}
\newcommand{\qg}[1]{Q^G_{#1}}
\nc{\pgg}{\pg{g}}
\nc{\qgg}{\qg{g}}

\nc{\plk}{\pl{\bus}}
\nc{\qlk}{\ql{\bus}}

\title{Theoretical Perspectives on Jabr-Type Convex Relaxations for AC Optimal Power Flow}

\author[1]{Gabor~Riccardi\thanks{Corresponding author. Email: \texttt{gabor.riccardi01@universitadipavia.it}}}
\author[1]{Ambrogio~Maria~Bernardelli}
\author[1]{Stefano~Gualandi}

\affil[1]{Department of Mathematics ``F. Casorati'', University of Pavia, 27100 Pavia, Italy}
\date{\vspace{-1.5em}}
\providecommand{\keywords}[1]{\textbf{Keywords} #1}
\begin{document}

\maketitle

\begin{abstract}
\noindent%
The alternating current optimal power flow problem is a fundamental yet highly nonconvex optimization problem whose structure reflects both nonlinear power flow physics and the topology of the underlying network. Among convex relaxations, the second-order cone relaxation introduced by Jabr has proven particularly influential, serving as a computationally efficient alternative to semidefinite relaxations and a foundation for numerous strengthening techniques.
In recent years, a variety of approaches have been proposed to tighten Jabr-type relaxations, including cycle-based constraints, convex envelopes of multilinear terms, and dual reformulations. However, these developments are often presented independently, concealing their common geometric and graph-theoretic foundations.
This paper provides a structured review of strengthening techniques for the Jabr relaxation and develops a unifying perspective based on multilinear equalities. We reinterpret cycle constraints as multilinear consistency conditions, analyze their convexification through classical convex hull theory, and investigate the relationship between primal McCormick relaxations and dual extended formulations. In particular, we identify structural conditions under which these relaxations coincide and clarify the distinction between convexifying the interaction graph and convexifying the feasible set of the ACOPF.
The resulting framework connects graph structure, multilinear convexification, and conic relaxations in a unified manner, offering both a conceptual synthesis of existing results and new insights for the design of stronger relaxations.
\end{abstract}

\section{Introduction}



The alternating current optimal power flow (ACOPF) problem, originally introduced in the early work of Carpentier \cite{carpentier1962contribution}, is a central model in electric power systems and a canonical example of a large-scale nonconvex network optimization problem. Its nonconvexity stems from the nonlinear AC power flow equations, which couple voltage magnitudes and phase angles via bilinear and trigonometric relations. The structural and computational challenges of ACOPF have been documented extensively; in particular, feasibility and optimization versions of the problem are strongly NP-hard \cite{NP_HARD, lehmann2015ac}. Comprehensive historical and technical accounts can be found in \cite{HistoryACOPF, ACOPF_Formulations, molzahn2019survey}.

Because of this intrinsic nonconvexity, convex relaxations have become fundamental tools for both theoretical analysis and global optimization approaches. Semidefinite programming (SDP) relaxations have received considerable attention due to their tightness and strong theoretical properties \cite{low2014convex, low2014convex2}. Variants and strengthened conic relaxations have further improved computational performance and bounds \cite{SDP, Oustry}. Surveys of conic relaxations and their properties are provided in \cite{zohrizadeh2020survey, molzahn2019survey}.

To address scalability concerns associated with SDP, second-order cone programming (SOCP) relaxations have emerged as attractive alternatives. In particular, the conic formulations introduced by Jabr \cite{Jabr_Relaxation, jabr2007conic, jabr2008optimal} provide a computationally efficient relaxation that preserves much of the network structure. Theoretical equivalence and relationships between different ACOPF formulations, including branch-flow and bus-injection models, are discussed in \cite{equivalence-i2-SOCP, farivar2013branch, farivar2013branch2}. Exactness results in radial networks have been established under various conditions \cite{gan2014exact, nick2017exact, huang2016sufficient}, highlighting the fundamental role played by network topology.

In meshed networks, however, the gap between SOCP relaxations and the nonconvex feasible set is closely tied to cycle consistency conditions. Over the past decade, several strengthening approaches have been proposed to address this issue. Cycle-based constraints and strong SOCP relaxations were introduced in \cite{cycle_Constr, kocuk2018matrix}. Additional enhancements include SDP-inspired cuts \cite{hijazi2016polynomial, miao2017least}, McCormick-based and quadratic convex relaxations \cite{weakQC, bynum2018strengthened}, and lifted nonlinear cuts \cite{bugosen2024applications}. These developments demonstrate that tightening the Jabr-type relaxation is intimately connected with capturing nonlinear relations induced by cycles in the network graph.

Although these techniques share common structural features, they are often developed in different frameworks, making their interrelations less clear. In particular, cycle consistency constraints can be interpreted as structured multilinear equalities, whose convexification relates directly to classical results on convex envelopes of multilinear functions \cite{McCormick1976, Rikun1997, trilinear_convex_envelope}. The strength of different relaxations then depends on the interaction graph induced by these multilinear terms and on how its convex hull is approximated \cite{Bao2009, leudke2012, DualIsBetterThanPrimal}.

\medskip
\noindent
\textbf{Contributions.}
This paper provides a structured review of strengthening techniques for the Jabr SOCP relaxation and develops a unifying perspective based on multilinear convexification and interaction graphs. Our contributions are as follows:

\begin{enumerate}
    \item We synthesize and organize existing strengthening approaches for Jabr-type relaxations, emphasizing their common structural foundations.
    \item We reinterpret cycle consistency conditions as multilinear equalities and connect them to classical convex hull results for multilinear functions.
    \item We analyze convexification strategies through the interaction graph of these multilinear relations, clarifying how network topology governs relaxation strength.
    \item We compare primal McCormick relaxations and dual extended formulations, identifying structural conditions under which these relaxations coincide.
    \item We distinguish between convexifying local graph-induced structures and convexifying the global ACOPF feasible set, thereby delineating the limits of cycle-based strengthening.
\end{enumerate}

By bridging graph structure, multilinear convexification, and conic relaxations, the paper offers both a conceptual synthesis of the literature and new structural insights for the design of stronger relaxations.

\medskip
\noindent
\textbf{Outline.}
The remainder of the paper is organized as follows.
Section~2 introduces the ACOPF formulation and the Jabr equality and SOCP relaxations.
%
Section~3 develops a unified perspective on strengthening techniques for the Jabr relaxation, interpreting cycle constraints as multilinear equalities, analyzing their convexification via interaction graphs, and comparing primal McCormick and dual extended formulations.
Section~4 concludes with implications and directions for further research.

\section{Formal Problem Statement and Jabr’s SOCP Relaxation}
In this section, we give the polar formulation of ACOPF and the formulation of the Jabr equality and inequality relaxations. Let us take a power network modeled as a digraph $\elecgraph = (\buses, \lines)$, where the set of nodes $\buses$ represents the set of buses and the set of arcs $\lines$ represents the set of transmission lines. Note that if a line exists between two buses $\bus$ and $\buss$, we have that both $\lin, \rlin \in \lines$. For every bus $\bus \in \buses$, we have a (possibly empty) set of electric power generators $\genset(\bus)$ and a certain electric power demand, also called load. The problem consists of meeting the energy demand at every bus, and doing so with the lowest possible energy generation cost. The solution must also obey Ohm's Law and Kirchhoff's Law and operational constraints. The physical characteristics of the network are described by the nodal admittance matrix $Y$, with components $Y_{\lin} = G_{\lin} + \mathrm{i}B_{\lin}$ for each line $\lin \in \lines$, and $G_{\bus\bus} = g_{\bus\bus} - \sum_{\buss \neq \bus}G_{\lin}$, $B_{\bus\bus} = b_{\bus\bus} - \sum_{\buss \neq \bus}B_{\lin}$, where $g_{\bus\bus}$ and $b_{\bus\bus}$ are the conductance and susceptance, respectively, for each bus $\bus \in \buses$, and $\mathrm{i}$ is the imaginary unit. We then have the complex voltage $\Vk = |\Vk|(\cos(\delta_{\bus}) + \mathrm{i}\sin(\delta_{\bus}))$ at each bus, active and reactive power output at the generator $g$ denoted by $\pgg$, $\qgg$, active and reactive power load at node $\bus$ denoted by $P^L_{\bus}$, $Q^L_{\bus}$, active and reactive power injected into branch $\lin$ denoted by $P_{\lin}$, $Q_{\lin}$. By dividing the physical laws, the power balances, and the variable limits into the real and imaginary components, we can write the following polar formulation:
\begin{subequations}\label{eq:ACOPF-polar}
\footnotesize
\begin{alignat}{3}
\min \;\;
  & \sum\limits_{g \in \genset} F_g(P^G_g) 
  & \;\; &  \label{DCOPF_obj2} \\
\text{s.t.} \;\;
  & \sum\limits_{\lin \in \lines} P_{\lin} + P^L_{\bus} - \sum\limits_{g \in \genset(\bus)} P^G_{g} = 0 
  & \;\; & \forall k \in \buses \label{Polar_flow_P} \\
& \sum\limits_{\lin \in \lines} Q_{\lin} + Q^L_{\bus} - \sum\limits_{g \in \genset(\bus)} Q^G_{g} = 0 
  & \;\; & \forall k \in \buses \label{Polar_flow_Q} \\
& P_{\lin} = G_{\bus\bus}|V_{\bus}|^2 + |V_{\bus}||V_{\buss}|(G_{\lin}\cos(\delta_{\bus}-\delta_{\buss}) - B_{\lin}\sin(\delta_{\bus}-\delta_{\buss})) 
  & \;\; & \forall \lin \in \lines \label{Polar_Elec_P} \\
& Q_{\lin} = -B_{\bus\bus}|V_{\bus}|^2 + |V_{\bus}||V_{\buss}|(-B_{\lin}\cos(\delta_{\bus} - \delta_{\buss}) + G_{\lin}\sin(\delta_{\bus} - \delta_{\buss})) 
  & \;\; & \forall \lin \in \lines\label{Polar_Elec_Q} \\
& P^{min}_{\bus} \leq P^G_{\bus} \leq P^{max}_{\bus} 
  & \;\; & \forall k \in \genset \label{Polar_Powerlim_P} \\
& Q^{min}_{\bus} \leq Q^G_{\bus} \leq Q^{max}_{\bus} 
  & \;\; & \forall k \in \genset \label{Polar_Powerlim_Q} \\
& (V_{\bus}^{min})^2 \leq |V_{\bus}|^2 \leq (V_{\bus}^{max})^2 
  & \;\; & \forall k \in \buses \label{Voltage_Limit}\\
& \theta_{\lin}^{min} \leq \delta_{\bus} - \delta_{\buss} \leq  \theta_{\lin}^{max}
  & \;\; & \forall k \in \buses \label{Polar_anglelimit} \\
  & P_{\lin}^2 + Q_{\lin}^2 \leq U_{\lin} 
  & \;\; & \forall \lin \in \lines \label{thermal_limit_polar}
\end{alignat}
\end{subequations}
where $F_g(P^G_g)$ is a linear or quadratic function and $\genset$ is the set of all generators, that is, $\genset \coloneqq \bigcup_{k \in \buses}\genset(\bus)$. Constraints~\eqref{Polar_flow_P}--\eqref{Polar_flow_Q} represent the flow conservation equation of the active and reactive power $P$ and $Q$. Constraints~\eqref{Polar_Elec_P}--\eqref{Polar_Elec_Q} are the expressions for the active and reactive power injected into branch $\lin$. Constraints~\eqref{Polar_Powerlim_P}--\eqref{Polar_anglelimit} represent generator active/reactive power limits, voltage magnitude limits, and angle difference limits, while Constraint~\eqref{thermal_limit_polar} corresponds to the thermal limits of the lines.

\subsection{Jabr equality relaxation}

We first consider a well-known variant of the Jabr relaxation model~\cite{Jabr_Relaxation}, derived from the polar ACOPF formulation. In this approach, the bilinear voltage terms are replaced by new variables:
\begin{align*}
    c_{\lin} = |V_{\bus}||V_{\buss}|\cos(\theta_{\lin}), & & 
    s_{\lin} = |V_{\bus}||V_{\buss}|\sin(\theta_{\lin}),
\end{align*}
where $\theta_{\lin} \coloneqq \delta_{\bus} - \delta_{\buss}$. From this definition, it follows that
\begin{equation*}
    |V_{\bus}|^2 = |V_{\bus}||V_{\bus}|\cos(\theta_{\bus\bus}),
\end{equation*}
which motivates the substitution
\begin{align*}
    c_{\bus\bus} = |V_{\bus}|^2, & &  s_{\bus\bus} = 0.
\end{align*}
To ensure consistency, we impose the following constraints based on properties of trigonometric functions and voltage magnitudes:
\begin{align*}
    c_{\bus\bus} \geq 0, & &
    c_{\lin} = c_{\rlin}, & & 
    s_{\lin} = -s_{\rlin}.
\end{align*}
Finally, by exploiting the identity
\begin{equation*}
    (|V_{\bus}||V_{\buss}|\cos(\theta_{\lin}))^2 + (|V_{\bus}||V_{\buss}|\sin(\theta_{\lin}))^2 = |V_{\bus}|^2|V_{\buss}|^2,
\end{equation*}
we introduce the quadratic constraint:
\begin{equation*}
    c_{\lin}^2 + s_{\lin}^2 = c_{\bus\bus}c_{\buss\buss}.
\end{equation*}

We obtain the following formulation
\begin{subequations} \label{eq:jabr-equality}
    \footnotesize
    \begin{alignat}{3}
      \min \;\;
  & \sum\limits_{g \in \genset} F_g(P^G_g) 
  & \;\; &   \\
\text{s.t.} \;\;
  & \sum\limits_{\lin \in \lines} P_{\lin} + P^L_{\bus} - \sum\limits_{g \in \genset(\bus)} P^G_{g} = 0 
  & \;\; & \forall k \in \buses \label{Power Flow Constraint J} \\
& \sum\limits_{\lin \in \lines} Q_{\lin} + Q^L_{\bus} - \sum\limits_{g \in \genset(\bus)} Q^G_{g} = 0 
  & \;\; & \forall k \in \buses \label{Power Flow Constraint J 2} \\
& P_{\lin} = G_{\bus\bus}c_{\bus\bus} + G_{\lin}c_{\lin} - B_{\lin}s_{\lin}
  & \;\; & \forall \lin \in \lines \label{c to P} \\
& Q_{\lin} = -B_{\bus\bus}c_{\bus\bus} - B_{\lin}c_{\lin} + G_{\lin}s_{\lin}
  & \;\; & \forall \lin \in \lines \label{c to Q} \\
& P^{min}_{\bus} \leq P^G_{\bus} \leq P^{max}_{\bus} 
  & \;\; & \forall k \in \genset  \label{Power generation Magnitude Constraint J} \\
& Q^{min}_{\bus} \leq Q^G_{\bus} \leq Q^{max}_{\bus} 
  & \;\; & \forall k \in \genset \label{Power generation Magnitude Constraint J 2}\\
& (V_{\bus}^{min})^2 \leq c_{\bus\bus} \leq (V_{\bus}^{max})^2 
  & \;\; & \forall k \in \buses \label{Voltage Magnitude Constraint J}\\
  & P_{\lin}^2 + Q_{\lin}^2 \leq U_{\lin} 
  & \;\; & \forall \lin \in \lines \label{Power Flow Magnitude Constraint J}\\
  & -\tan(\eta_{\lin})\,c_{\lin} \le s_{\lin} \le \tan(\eta_{\lin})\,c_{\lin}
& \;\; & \forall \lin \in \lines \\
& c_{\lin} \ge 0 
& \;\; & \forall \lin \in \lines \label{Jabr voltage diff contraint}\\
  & c_{\bus\bus} \geq 0, & \;\; &\forall k \in \buses \\
  & c_{\lin} = c_{\rlin}, \; s_{\lin} = - s_{\rlin}, & \;\; &\forall \lin \in \lines \\
  & c_{\lin}^2 + s_{\lin}^2 = c_{\bus\bus}c_{\buss\buss} & \;\; & \forall \lin \in \lines \label{constr:jabr-eq}
    \end{alignat}
\end{subequations}
%
%
which we will refer to as the \emph{Jabr equality ACOPF relaxation}.

This relaxation is in general not exact, as a feasible solution for this model may not be mapped to a feasible solution of~\eqref{eq:ACOPF-polar}. In~\cite{Jabr_Relaxation}, the corresponding feasibility problem is shown to be exact when the underlying network is a tree, i.e., when the system is radial with possibly multiple sources. In this section, we provide an explicit proof of this exactness result.

\begin{proposition}
\label{Jabr equality model exactness}
    If $(\buses,\lines)$ is a multisource radial network, then the Jabr equality ACOPF relaxation~\eqref{eq:jabr-equality} is exact. 
\end{proposition}

\begin{proof}
Let $P_g^G,Q_g^G, c_{\lin}, s_{\lin}$ be an optimal solution to the Jabr equality ACOPF relaxation~\eqref{eq:jabr-equality}. Let $|V_{\bus}| \coloneqq \sqrt{c_{\bus\bus}}$. Since $\frac{c_{\lin}^2}{c_{\bus\bus}c_{\buss\buss}}+\frac{s_{\lin}^2}{c_{\bus\bus}c_{\buss\buss}} = 1$  then there exists $\theta_{\lin}$ such that $\frac{c_{\lin}^2}{c_{\bus\bus}c_{\buss\buss}} = \cos^2(\theta_{\lin})$ and $\frac{s_{\lin}^2}{c_{\bus\bus}c_{\buss\buss}} = \sin^2(\theta_{\lin})$. We can fix a reference bus $r \in \buses$ for which we have $\delta_r = 0 \mod 2\pi$. We need to define the voltage angle $\delta_{\bus}$ for every $\bus \in \buses$ such that
\begin{equation} \label{exactness of theta}
    \theta_{\lin}=\delta_{\bus} - \delta_{\buss}\mod 2 \pi \quad  \forall \lin \in \lines .
\end{equation}
 For each $\bus \in \buses$ let $(k_1,...,k_n)$ be a path in the graph where $\bus_1 = r$ and $\bus_n = k$. We define $\delta_h \coloneqq  \sum_{i=1}^{n-1}\theta_{k_{i+1}k_{i}} \mod 2\pi$. The various voltage angles are well defined because there are no cycles in the graph and therefore there exist only one path connecting $\bus$ and $h$. Furthermore, with this definition,~\eqref{exactness of theta} holds by construction. This is a feasible solution for the ACOPF problem in polar form~\eqref{eq:ACOPF-polar} because the magnitude constraints~\eqref{Power generation Magnitude Constraint J}--\eqref{Power Flow Magnitude Constraint J} and the Power Balance constraints~\eqref{Power Flow Constraint J}--\eqref{Power Flow Constraint J 2} remain the same. Moreover, since with this definition of $|V_{\bus}|$ and $\delta_{\bus}$ we have that $c_{\lin}= |V_{\bus}||V_{\buss}|\cos{\theta_{\lin}}$ and $s_{\lin}= |V_{\bus}||V_{\buss}|\sin{\theta_{\lin}}$, then equations~\eqref{c to P}--\eqref{c to Q} imply that~\eqref{Polar_flow_P}--\eqref{Polar_flow_Q} hold. Lastly, \eqref{Jabr voltage diff contraint} implies the voltage angle constraint \eqref{Polar_anglelimit}. Since the Jabr equality ACOPF~\eqref{eq:jabr-equality} is a relaxation of the Polar ACOPF model~\eqref{eq:ACOPF-polar}, $|V_{\bus}|$ and $\delta_{\bus}$ together constitute an optimal solution of~\eqref{eq:ACOPF-polar}.
\end{proof}

 \subsection{Jabr--SOCP relaxation}
 While model~\eqref{eq:jabr-equality} yields an exact relaxation for radial networks, it is not convex and hard to solve. A convex relaxation can then be obtained by substituting constraint~\eqref{constr:jabr-eq} with:
 \begin{equation} \label{constr:jabr-ineq}
     c_{\lin}^2 + s_{\lin}^2 \leq c_{\bus\bus}c_{\buss\buss}.
 \end{equation}
 We refer to this relaxation as the \emph{Jabr--SOCP relaxation}, which is due to the fact that we can rewrite~\eqref{constr:jabr-ineq} as
 \begin{equation*}
     c_{\lin}^2 + s_{\lin}^2 + \left(\frac{c_{\buss\buss}-c_{\bus\bus}}{2} \right)^2 \leq \left(\frac{c_{\buss\buss}+c_{\bus\bus}}{2} \right)^2,
 \end{equation*}
which is a rotated second order cone in $\mathbb{R}^4$.

Building upon his earlier work on convex formulations for radial distribution systems~\cite{Jabr_Relaxation}, Jabr extended the conic approach to general transmission networks. In~\cite{jabr2007conic}, he proposed an SOCP-based formulation of the AC load-flow equations for meshed networks, introducing additional trigonometric constraints to maintain voltage-angle consistency around network cycles. Jabr incorporated linearized approximations of the angle relationships, effectively replacing the nonlinear arctangent constraint with successive affine approximations. The resulting algorithm solves a sequence of conic programs, each updating the linearization points until voltage-angle consistency is achieved. This approach preserved the computational advantages of conic programming while extending its applicability to large transmission networks.  

In a subsequent paper, Jabr~\cite{jabr2008optimal} integrated this extended conic quadratic (ECQ) representation into the OPF problem. The ECQ-OPF formulation includes common control devices such as tap-changing and phase-shifting transformers and unified power-flow controllers, while retaining the linearity of power injection equations. All nonlinearities are confined to a small set of rotated conic and arctangent constraints, which makes the problem suitable to efficient solution by primal–dual interior-point methods. The structure also supports LP-type scaling techniques that improve numerical conditioning. Tests on large benchmark systems confirmed that the ECQ-OPF achieves accurate and stable results for both economic dispatch and loss minimization, demonstrating the practicality of conic quadratic formulations for large-scale power system optimization.  

Although these formulations yield convex and computationally efficient models, they are generally not exact, as the relaxed feasible region can include points that do not correspond to physically realizable AC states. This limitation motivated extensive research into characterizing conditions for exactness and strengthening the SOCP relaxation, topics we review next.  

Following Jabr’s work, a large body of research examined the conditions under which the SOCP relaxation is exact, i.e., when its solution coincides with that of the nonconvex ACOPF. These conditions are of theoretical importance, since exactness of the SOCP relaxation implies exactness of tighter formulations such as SDP, QC, or strong SOCP relaxations~\cite{low2014convex2}. However, because ACOPF remains NP-hard~\cite{NP_HARD} even for radial networks~\cite{lehmann2015ac}, most results apply only under restrictive or idealized assumptions.

Many of these studies rely on the load oversatisfaction assumption, which allows active and reactive power demands to increase arbitrarily. This facilitates convexity and proof of exactness, but does not reflect realistic load behavior where demand is fixed or voltage-dependent~\cite{nick2017exact,christakou2017ac}. Within this theoretical framework, Gan and coauthors~\cite{gan2014exact} showed that, after a mild reformulation of the OPF equations for radial networks, the corresponding SOCP relaxation becomes exact under verifiable conditions on line parameters. Huang and coauthors~\cite{huang2016sufficient} proved that exactness is verified in radial networks for which there is either no reverse power flow or reverse power flow that consists only of reactive power or only of active power, while Nick and coauthors~\cite{nick2017exact} extended the result to include line shunt components. Chen and coauthors~\cite{chen2018socp} proposed an over-satisfaction condition that guarantees exactness in radial networks, while in~\cite{sojoudi2013convexification} the case of meshed networks is analyzed with many controllable phase-shifting transformers and a load over-satisfaction assumption.

In parallel, Farivar and Low~\cite{farivar2013branch,farivar2013branch2} developed the branch-flow model, which expresses power flow variables in terms of branch currents and squared voltages, admitting convex relaxations that are globally optimal for certain radial systems. Subsequent analyses, notably by Christakou and coauthors~\cite{christakou2017ac}, demonstrated that several of these assumptions (particularly those excluding voltage magnitude or current limits) do not hold in practical networks, highlighting intrinsic limits of such convexifications.

For different SOCP formulations and their equivalence, see~\cite{low2014convex}; for a broader study of exactness and distributed algorithms, see~\cite{low2014convex2}; and for a comprehensive survey of conic relaxations for OPF, see~\cite{zohrizadeh2020survey}.

In addition to these exactness results, much of the literature has focused on strengthening the SOCP relaxation, a topic we return to in \Cref{sec:tightening} together with our own contributions. Before that, we review linearization approaches, which provide an alternative direction for handling the ACOPF problem.

\subsection{Linearization approaches}
An alternative line of research to convex relaxations such as SOCP or SDP is the development of linear relaxations and linearization approaches for ACOPF. The motivation is that linear programs scale extremely well and can be directly embedded into MILP frameworks, although at the cost of reduced tightness compared to conic relaxations.

Early works by Bienstock and Muñoz~\cite{bienstock2014linear,bienstock2015approximate} introduced one of the first systematic polyhedral relaxations of ACOPF. Their approach operates in a lifted variable space and approximates the nonconvex quadratic and trigonometric relationships between voltages and power flows through a collection of convex inequalities. In particular, they propose $\Delta$-inequalities, loss inequalities, and circle inequalities.
The $\Delta$-inequalities bound the active power flow on each line in terms of voltage magnitude differences between connected buses, obtained by introducing lifted variables representing real and imaginary voltage components.
The loss inequalities approximate the quadratic expression for active-power losses along each branch by relaxing the corresponding equality into a conic inequality involving voltage-difference variables.
The circle inequalities describe the feasible $(P,Q)$ region of complex power flows as a circle in the plane, arising from the relationship between voltage magnitude, current, and apparent power. These can be written as rotated SOCP constraints but are further linearized using tangent planes, yielding a fully linear outer approximation of the AC feasible region.

In addition, in~\cite{bienstock2018lp} they proposed a theoretically grounded family of linear programs capable of approximating, to arbitrary accuracy, the solutions of nonlinear power system optimization problems, including those with integer variables. Their construction exploits network sparsity by leveraging the graph-theoretic treewidth: the size of the LP formulation grows linearly with the number of buses, logarithmically with the approximation precision, and exponentially with the treewidth.
This provides strong representability and convergence guarantees, showing that any OPF problem over a bounded-treewidth network can, in principle, be represented by a finite LP of controllable accuracy. However, the exponential dependence on treewidth constrains practical applicability to small or nearly radial networks, and empirical validation on realistic systems remains limited.

A further refinement was introduced by Bienstock and Villagra~\cite{BIENSTOCK}, in which they developed very tight and numerically stable linearly constrained relaxations. Their approach is based on a cutting-plane framework that iteratively approximates the conic constraints of the SOCP relaxation through dynamically generated linear cuts, while employing rigorous cut-management techniques to maintain numerical stability and control model size. The resulting method produces high-quality linear relaxations that can be constructed and solved efficiently with standard LP solvers, offering tight lower bounds at a fraction of the computational cost of nonlinear convex relaxations.  
The authors provide a theoretical justification for the effectiveness of these linear relaxations by showing that the active-power loss inequalities originally proposed in~\cite{bienstock2014linear} act as outer-cone envelope approximations of the SOCP constraints. They argue that such constraints are essential to achieving tightness, as they implicitly ensure that every unit of demand and loss is matched by a corresponding unit of generation, via a flow-decomposition argument.  
A key strength of their framework lies in its robustness and scalability. The cutting-plane algorithm features a warm-start mechanism that allows previously generated cuts to be reused when solving related OPF instances, such as those arising in multi-period formulations or perturbed system conditions. This warm-start capability yields significant computational savings and makes the method particularly well suited for time-series or multi-period ACOPF problems, where traditional nonlinear solvers struggle to converge.  
Extensive numerical experiments demonstrate that the proposed linear relaxations provide bounds comparable to or tighter than state-of-the-art nonlinear convex approaches, while solving significantly faster. Moreover, the authors highlight that the linearized relaxations, unlike SOCPs, offer stronger theoretical guarantees on bounding quality when paired with convex quadratic objectives.

Other approaches aim at directly linearizing the AC power-flow equations themselves. Coffrin and Van Hentenryck~\cite{coffrin2014linear} proposed the Linear Programming AC (LPAC) approximation, which extends the classical DC model by incorporating voltage magnitudes, reactive power, and active losses into a tractable linear framework. The LPAC formulation introduces linear approximations for both the sine and cosine terms of the AC power flow equations. In particular, the cosine function $\cos(\theta_i - \theta_{\bus})$ is represented by a polyhedral convex envelope built from a set of equally spaced tangent line segments, forming a piecewise-linear outer approximation of the trigonometric nonlinearity.  
Three variants are presented to suit different levels of available data. The hot-start LPAC assumes fixed voltage magnitudes and linearizes power flows around these known values. The warm-start LPAC introduces small voltage-magnitude deviations as decision variables and uses first-order Taylor expansions to model reactive power, allowing limited flexibility in voltages. The cold-start LPAC assumes nominal voltage magnitudes (1~p.u. at load buses) and fixed generator voltages, enabling initialization without a prior AC solution. Across all versions, the active and reactive power flows $(P_{ik}, Q_{ik})$ are modeled as affine functions of voltage magnitudes and phase angles, with a convex-envelope variable $\psi_{ik}$ representing the cosine term.  
By maximizing $\psi_{ik}$ within the convex envelope, the LPAC formulation minimizes relaxation error and yields a linear program that jointly represents real and reactive power, voltage magnitudes, and losses.

Fortenbacher and Demiray~\cite{fortenbacher2019linear} developed novel LP and QP formulations for the optimal power flow problem based on linearized power-flow equations and absolute-value loss approximations. Their methods link the full decision-variable domain to linear power flow models while explicitly representing active power losses through linear or piecewise-linear absolute-value terms.
Unlike earlier LP formulations that were restricted to radial or simplified networks, their approach is generalizable to arbitrary grid topologies and voltage levels, including transmission, distribution, and low-voltage systems.
A key feature of their formulation is that it can be solved in a single optimization run using standard LP/QP solvers, avoiding the iterative subproblem structure required in earlier works, and it addresses the issue of fictitious losses with a minimal set of constraints, without adding penalty terms to the objective.
Simulation studies demonstrate that the proposed LP/QP methods yield feasible AC solutions whenever approximate solutions exist, with significantly lower computational effort than full nonlinear ACOPF models and only moderate deviations in objective value for typical benchmark systems.

Sadat and Sahraei-Ardakani~\cite{sadat2022tuning} proposed a successive linear programming (SLP) algorithm for solving the ACOPF problem, building upon active-set optimization principles. Active-set methods identify and monitor only the most influential constraints at each iteration, thereby reducing search complexity compared to nonlinear interior-point solvers. Although more sensitive to initialization, this feature can be exploited effectively for time-series OPF studies, where good starting points are readily available.  
In their framework, the nonlinear ACOPF is iteratively linearized into an LP subproblem using first-order Taylor expansions of the nonlinear constraints and a piecewise-linear interpolation of the cost function. The resulting LP is solved repeatedly until violations in the original nonlinear equations fall within a specified tolerance. The approach offers several advantages: it can leverage fast, reliable commercial LP solvers; it easily integrates with existing energy management systems that prefer linear formulations; and it allows the use of DCOPF solutions as high-quality initializations, substantially reducing computational effort.  
A key innovation of their work lies in the introduction of tuning methods designed to ensure convergence of the linearized subproblems. These include the addition of random perturbations to generator cost functions to avoid degeneracy, the enforcement of convergence and wrap-around power-balance constraints, and the introduction of penalties to maintain voltage magnitudes near unity.

\section{Insights on Tightening the Jabr Relaxation}\label{sec:tightening}

In the previous section, we reviewed several methods proposed in the literature for efficiently implementing the Jabr-SOCP relaxation. However, as previously noted, even the Jabr equality ACOPF relaxation~\eqref{eq:jabr-equality} is not exact on meshed networks. This has motivated a large body of work on strengthening techniques that use reformulations, envelopes, or cut generation to reduce the relaxation gap.

Bynum et al.~\cite{bynum2018strengthened} proposed the use of McCormick envelopes for bilinear terms, combined with Optimization-Based Bound Tightening (OBBT), showing substantial reductions in relaxation gaps. More recent contributions include lifted nonlinear cuts (LNCs)~\cite{bugosen2024applications}, accuracy-enhancing cutting-plane methods~\cite{gholami2022accuracy}, and matrix minor reformulations~\cite{kocuk2018matrix}, all of which improve tightness while maintaining tractability. Guo, Nagarajan, and Bodur~\cite{guo2025tightening} further extend these ideas in the context of quadratic convex relaxations for transmission switching, highlighting the role of network topology in relaxation strength.

Another important line of research for strengthening SOCP relaxations is built on cycle-based constraints, which we explore in the next section.

\subsection{Cycle-based constraints}

As done in the proof of \Cref{Jabr equality model exactness}, given a solution to the Jabr relaxation we can always recover $|V_{\bus}|,\theta_{\lin}$ such that: $c_{\bus\bus} = |V_{\bus}|^2, \; c_{\lin} = |V_{\bus}||V_{\buss}|\cos{\theta_{\lin}} $ and $s_{\lin} = |V_{\bus}||V_{\buss}|\sin{\theta_{\lin}}$.
In general, the model is exact whenever at each node $\bus \in \buses$ we can define $\delta_{\bus}$ such that for all $(k,m) =l \in \lines$ we have: $\theta_{\lin} = \delta_{\bus} - \delta_{\buss} \mod 2\pi$. This is possible whenever the following cycle condition holds for every cycle $(k_1,...,k_n,k_{n+1}=k_1)$ in the graph: 
\begin{equation} \label{eq: cycle condition}
    \sum_{i=1}^{n}\theta_{k_{i+1}k_i}=0  \mod 2\pi.
\end{equation}
 Let  $\delta_h$ be defined as in the proof of Proposition~\ref{Jabr equality model exactness}. Then, if $(k_1=r,...,k_n=k)$ and $(h_1=r,...,h_n=k)$ are two paths from $r$ to $\bus$, we have
 \[
 \delta_h =  \sum_{i=1}^{n-1}\theta_{k_{i+1}k_{i}} =  \sum_{i=1}^{n-1}\theta_{h_{i+1}h_{i}} \mod 2\pi
 \]
 for the cycle condition. So $\delta_h$ is well defined and we obtain an optimal feasible solution for the Polar ACOPF model~\eqref{eq:ACOPF-polar}.

We can check that this condition holds directly on the variables $c_{\lin}$ and $s_{\lin}$ since $ \sum_{i=1}^{n} \theta_{k_{i+1}k_i} = 0 \mod 2\pi \iff \cos(\theta_{k_1k_n} + \sum_{i=1}^{n-1} \theta_{k_{i+1}k_i}) = 1$ and thanks to the following generalized sum-to-product formula.

\begin{proposition}[Generalized sum-to-product formula]
\label{sum-to-product}
    The following identities hold:
    \begin{align}
        \label{sin1}
        \sin\left(\sum_{i=1}^n\theta_i\right) &=\sum_{\ell = 0}^{\lfloor n/2 \rfloor}\sum_{\substack{A \subset [n]\\|A|=2\ell+1}}(-1)^\bus\prod_{h \in A}\sin(\theta_h)\prod_{h \in A^c}\cos(\theta_h), \\
        \label{cos1}
        \cos\left(\sum_{i=1}^n\theta_i\right) &=\sum_{\ell = 0}^{\lfloor n/2 \rfloor}\sum_{\substack{A \subset [n]\\|A|=2\ell}}(-1)^\bus\prod_{h \in A}\sin(\theta_h)\prod_{h \in A^c}\cos(\theta_h).
    \end{align}
\end{proposition}

\begin{proof}
Let $\Im(x)$ denote the imaginary part of $x \in \mathbb{C}$. Using Euler's identity and since $\Im(i^{\ell})$ is zero when $\bus$ is even, we obtain:
\begin{align*}
\sin\left(\sum_{j=1}^n\theta_j\right) &= \Im(\exp(i\sum_{j=1}^n\theta_j)) \\
&= \Im(\prod_{j=1}^n\exp(i\theta_j)) \\
&= \Im(\prod_{j=1}^n(\cos(\theta_j)+i\sin(\theta_j))) \\
&= \Im(\sum_{A \subset [n]} \prod_{h \in A}(i\sin(\theta_h))\prod_{h \in A^c}(\cos(\theta_h))) \\
&= \Im(\sum_{A \subset [n]} i^{|A|} \prod_{h \in A}(\sin(\theta_h))\prod_{h \in A^c}(\cos(\theta_h))) \\
&= \sum_{\ell = 0}^{\lfloor n/2 \rfloor}\sum_{\substack{A \subset [n]\\|A|=2\ell+1}}(-1)^\bus\prod_{h \in A}\sin(\theta_h)\prod_{h \in A^c}\cos(\theta_h).
\end{align*}

\end{proof}

We note that for the constraint \eqref{Voltage_Limit}, the voltage $|V_{\bus}|$ is nonzero at each bus $\bus\in \buses$. By multiplying \eqref{cos1}, applied to $\sum_{i=1}^{n}\theta_{k_{i+1}k_i}$, by $\prod_{h=1}^n|V_{k_h}|^2$ and using the definition of the variables $c$ and $s$ we obtain the following observation.

\begin{observation}[cycle constraint]\label{obs: cycle constraint}
    Let $(k_1,...,k_n,k_1)$ be a cycle in the graph. Then $\sum_{i=1}^{n}\theta_{k_{i+1}k_i}=0\mod 2\pi$, where $n+1=1$, if the voltages $|V_i|\neq 0$ for each bus and the following equation holds:
    \begin{equation}\label{eq: cycle constraint}
        \sum_{k = 1}^{\lfloor n/2 \rfloor}\sum_{\substack{A \subset [n]\\|A|=2k}}(-1)^\bus\prod_{h \in A}s_{k_hk_{h+1}}\prod_{h \in A^c}c_{k_hk_{h+1}}=\prod_{k=1}^nc_{k_i,k_i}.
    \end{equation}
\end{observation}

Thus, if the cycle constraint holds for every cycle in the network for optimal solution of the Jabr equality Equation \eqref{eq:jabr-equality}, then it is an optimal solution to the OPF problem Equation $\eqref{eq:ACOPF-polar}$. However, in real-life applications, the number of cycles in a graph can be very large for large networks, making this result impractical; it is sufficient instead to enforce it on the cycles of a cycle basis of the graph~\cite{schrijver-book}.

Various works have developed efficient methods to implement cycle-based constraints. Kocuk, Dey, and Sun~\cite{cycle_Constr} were among the first to demonstrate the effectiveness of such constraints, showing that they substantially reduce the SOCP–SDP gap on meshed networks. Their work was later extended to mixed-integer settings~\cite{kocuk2017new}, proposing strong MISOCP relaxations for optimal transmission switching problems. Guo, Nagarajan, and Bodur~\cite{guo2025tightening} incorporated cycle-based polynomial constraints into quadratic convex relaxations, further improving bounds through disjunctive cutting-plane methods. Related techniques, such as SDP-inspired affine inequalities~\cite{miao2017least} and polynomial cuts~\cite{hijazi2016polynomial}, are based on the idea of reinforcing cycle feasibility. 

An alternative approach to enforce cycle-based constraints is proposed by  Shao et al.~\cite{Shao2025}, who develop an equivalent reformulation of the Jabr--SOCP relaxation by expressing phase angle differences through inverse trigonometric functions. This reformulation makes cycle consistency implicit: Kirchhoff's voltage law is enforced by requiring that the sum of angle differences along each cycle in a cycle basis vanishes. Building on this, they introduce a new SOCP relaxation combined with reverse cone cuts applied adaptively to a small subset of edges, and propose a custom OBBT scheme achieving tight optimality gaps.

Together, these works establish cycle-based constraints as one of the most powerful strategies for tightening SOCP relaxations, often approaching the strength of SDP while retaining scalability.

\subsection{Decomposition of cycle constraint into 3- and 4-cycle constraints}
Modeling the convex relaxation of the cycle constraint \eqref{eq: cycle constraint} would mean adding an exponential number of constraints (resp. variables) for the primal (dual) formulation with respect to the length of the cycle. For this reason in \cite{cycle_Constr}, the authors decomposed the single cycle constraints in multiple cycle constraints on smaller cycles as follows:

Consider the circuit
\[
\circuit = (\vertex_1,\arc_1,\vertex_2,\arc_2,\dots,\arc_{\sizecircuit},\vertex_{\sizecircuit+1}=\vertex_1).
\]
We obtain a decomposition of \(\circuit\) by partitioning its arcs into elements that are either single 1-paths (single arcs) or \(2\)-paths.
For a given element \(\Path\) of the partition, with initial and terminal vertices \(\ivertex\) and \(\evertex\), we define the auxiliary cycle \(\cycle_{\Path}\) as the concatenation of:
\begin{enumerate}
  \item the auxiliary arc \(\aarc_{\ivertex} = (\vertex_1, \ivertex)\),
  \item the element \(\Path\),
  \item the auxiliary arc \(\aarc_{\evertex} = (\evertex, \vertex_1)\).
\end{enumerate}

The resulting cycle \(\cycle_{\Path}\) is a 3-cycle if \(\Path\) consists of a single arc, and a 4-cycle if \(\Path\) is a 2-path. See \Cref{fig:decomposition} for an example. The sum of angle differences in \eqref{eq: cycle condition} over the original circuit \(\cycle\) is zero if and only if the sum of angle differences is zero over each auxiliary cycle \(\cycle_{\Path}\) associated with the elements \(\Path\) of the partition. Hence, it suffices to consider cycle constraints over 3- and 4-cycles.

For instance, consider the 3-cycle \(\cycle = (1,(1,2),2,(2,3),3,(3,1),1)\). The cycle constraint \eqref{eq: cycle constraint} applied to \(\cycle\) is:
\begin{equation}\label{eq: 3 cycle constraint}
    p_3 \coloneqq c_{12}(c_{23}c_{31} - s_{23}s_{31}) - s_{12}(s_{23}c_{31} + c_{23}s_{31}) - c_{11}c_{22}c_{33} = 0.
\end{equation}

Similarly, for the 4-cycle \(\cycle = (1,(1,2),2,(2,3),3,(3,4),4,(4,1))\), the cycle constraint becomes:
\begin{equation}\label{eq: 4 cycle constraint}
\begin{aligned}
    p_4 \coloneqq &(c_{12}c_{34} - s_{12}s_{34})(c_{23}c_{41} - s_{23}s_{41})
    \\& - (s_{12}c_{34} + c_{12}s_{34})(s_{23}c_{41} + c_{23}s_{41})
    - c_{11}c_{22}c_{33}c_{44} = 0.
\end{aligned}
\end{equation}

For both 3- and 4-cycles, the cycle constraint can be equivalently expressed as a system of two multilinear equations, each of degree two.  
Consider the 4-cycle case: equation \eqref{eq: cycle condition} can be rewritten as
\[
\theta_{12} + \theta_{34} = -(\theta_{23} + \theta_{41}) \mod 2\pi,
\]
which is equivalent to the following system:
\begin{subequations}\label{eq: splitted cycle contraint}
    \begin{align}
    \sin(\theta_{12} + \theta_{34}) &= \sin(-\theta_{23} - \theta_{41}), \\
    \cos(\theta_{12} + \theta_{34}) &= \cos(-\theta_{23} - \theta_{41}).
\end{align}
\end{subequations}
Expanding using the sum-to-product formulas and multiplying both equations by \(\prod_{i=1}^{4}\abs{\V{i}}\), then substituting the cosine and sine terms with \(c_{ij}\) and \(s_{ij}\), we obtain:
\begin{subequations}\label{eq: 4 cycle constr var}
    \begin{align}
    q^4_1 &\coloneqq s_{12}c_{34} + c_{12}s_{34} + s_{23}c_{41} + c_{23}s_{41} = 0, \\
    q^4_2 &\coloneqq  c_{12}c_{34} - s_{12}s_{34} - c_{23}c_{41} + s_{23}s_{41}=0.
\end{align}
\end{subequations}

Analogously, for the 3-cycle constraint \eqref{eq: 3 cycle constraint}, we can rewrite it as:
\begin{subequations}\label{eq: 3 cycle constr var}
    \begin{align}
    q^1_{3} &= s_{12}c_{33} + c_{23}s_{31} + s_{23}c_{31} \\
q^2_{3} &= c_{12}c_{33} - c_{23}c_{31} + s_{23}s_{31}.
\end{align}
\end{subequations}
These relations comprise bilinear functions which satisfy the assumption of Proposition~\ref{prop: multilin convex hull}, for which the convex hull can be obtained by introducing auxiliary variables for each bilinear term and enforcing the corresponding McCormick envelopes. This is precisely the strategy employed in~\cite{cycle_Constr}.

In~\cite{guo2025tightening}, a different approach is used. Instead of substituting the terms \(\barra{\V{i}}\barra \barra{\V{j}}\barra \cos(\delta_{ij})\) with variables \(c_{ij}\), the authors introduce a single variable representing \(\cos(\theta_{ij})\) and derive 3- and 4-cycle constraints analogous to \eqref{eq: 3 cycle constr var} and \eqref{eq: 4 cycle constr var} using sum-to-product formulas.  

More generally, several equivalent formulations of cycle constraints are possible, only a subset of which is reported here. Additional equivalent reformulations can be obtained through different decompositions of the cycles in the cycle basis, alternative choices of the cycle basis itself, or permutations of the angle variables in \eqref{eq: splitted cycle contraint}. In~\cite{guo2025tightening}, multiple relaxations corresponding to such equivalent reformulations of the cycle constraints are implemented simultaneously in the model. This strategy leads to tighter bounds than applying the relaxation to a single formulation of the cycle constraints. Furthermore, instead of applying McCormick relaxations to each bilinear term individually, they employ the dual formulation of the convex envelope, which provides tighter bounds. These outcomes are not straightforward and merit further examination. To understand this question better, we provide a small review of results on the convexification of multilinear functions.

To conclude this section, Table \ref{tab:method-comparison} summarizes and compares the various methods employed in the literature leveraging cycle constraints to solve ACOPF. ``Primal'' refers to McCormick-type linearization of each bilinear term separately; ``Dual'' refers to the extended-formulation representation of the convex hull via vertex weights $\lambda_v$. \emph{Cycle form} indicates how the cycle constraints are encoded before relaxation. In Table \ref{tab:numerical-comparison} each entry is the gap reported after applying the cycle-based strengthening in the root node (no spatial branching). ``--'' indicates the case was not reported in that paper.



\begin{table}[ht!]
\centering
\caption{Methodological comparison of cycle-based strengthening approaches
         for the Jabr--SOCP relaxation of ACOPF.
         }
\label{tab:method-comparison}
\setlength{\tabcolsep}{4pt}
\renewcommand{\arraystretch}{3}
\resizebox{\textwidth}{!}{%
\begin{tabular}{@{}lllllll@{}}
\toprule
\textbf{Reference}
  & \textbf{Base relax.}
  & \textbf{Cycle form}
  & \textbf{Convexification}
  & \textbf{P/D}
  & \textbf{Cycle basis}
  & \textbf{Int.\ ext.} \\
\midrule

\makecell[l]{Kocuk et al.~\cite{cycle_Constr} \\ (2016)}
  & SOCP
  & \makecell[l]{3- \& 4-cycle bilinear \\ eqs.\ \eqref{eq: 3 cycle constr var}--\eqref{eq: 4 cycle constr var}}
  & McCormick envelopes
  & P
  & \makecell[l]{Yes \\ (fundamental)}
  & No \\[2pt]

\makecell[l]{Kocuk et al.~\cite{kocuk2018matrix} \\ (2018)}
  & SOCP
  & \makecell[l]{Matrix minor \\ constraints}
  & \makecell[l]{McCormick + \\ reverse-cone cuts}
  & P
  & Yes
  & No \\[2pt]

\makecell[l]{Kocuk~\cite{kocuk2017new} \\ (2017)}
  & MISOCP
  & \makecell[l]{3- \& 4-cycle \\ bilinear eqs.}
  & McCormick envelopes
  & P
  & Yes
  & \makecell[l]{Yes \\ (OTS)} \\[2pt]

\makecell[l]{Bynum et al.~\cite{bynum2018strengthened} \\ (2018)}
  & SOCP
  & \makecell[l]{Rectangular form +\\ ref.\ bus cuts}
  & McCormick + OBBT
  & P
  & Implicit
  & No \\[2pt]

\makecell[l]{Hijazi et al.~\cite{hijazi2016polynomial} \\ (2016)}
  & SDP / SOCP
  & Polynomial SDP cuts
  & \makecell[l]{SDP lifting + \\ linearization}
  & P
  & Partial
  & No \\[2pt]

\makecell[l]{Miao et al.~\cite{miao2017least} \\ (2017)}
  & SOCP
  & \makecell[l]{SDP-inspired \\ affine cuts}
  & \makecell[l]{Affine outer \\ approximation}
  & P
  & Partial
  & No \\[2pt]

\makecell[l]{Guo et al.~\cite{guo2025tightening} \\ (2025)}
  & MISOCP
  & \makecell[l]{Multiple equiv.\ \\ bilinear reformulations}
  & \makecell[l]{Dual convex hull \\ (Benders cuts)}
  & D
  & Yes
  & \makecell[l]{Yes \\ (OTS)} \\[2pt]

\makecell[l]{Shao et al.~\cite{Shao2025} \\ (2025)}
  & SOCP
  & \makecell[l]{Arctan / angle-sum \\ (cycle basis)}
  & \makecell[l]{McCormick + \\ reverse-cone + OBBT}
  & P
  & \makecell[l]{Yes \\ (exact reform.)}
  & No \\

\bottomrule
\end{tabular}
}
\smallskip

\noindent\footnotesize
\textbf{P/D}: Primal (McCormick-type) or Dual (extended formulation) convexification.
\textbf{Int.\ ext.}: Integer extension (OTS = optimal transmission switching).
\end{table}

\begin{table}[ht!]

\centering
\caption{Comparison of optimality gaps (\%) on common PGLIB/MATPOWER benchmark instances.}
\label{tab:numerical-comparison}
\scriptsize
\setlength{\tabcolsep}{5pt}
\renewcommand{\arraystretch}{1.3}
\begin{tabular}{@{}lrrrrr@{}}
\toprule
\textbf{Case}
  & \makecell{\textbf{Kocuk}\\\textbf{et al.~\cite{cycle_Constr}}\\\textbf{(2016)}}
  
  & \makecell{\textbf{Bynum}\\\textbf{et al.~\cite{bynum2018strengthened}}\\\textbf{(2018)}}
  & \makecell{\textbf{Guo}\\\textbf{et al.~\cite{guo2025tightening}}\\\textbf{(2025)}}
  & \makecell{\textbf{Shao}\\\textbf{et al.~\cite{Shao2025}}\\\textbf{(2025)}} \\
\midrule
case9   & 0.00  & 0.00 & 0.1 & --   \\
 case14  & 0.06  & 0.06 & -- & 0.00 \\
case30  & 0.34  & 0.08 & 11.0  & 0.07 \\
case57  & 0.01  & 0.08 & 0.1   & 0.06 \\
case118 & 0.16  & 0.06 & 0.3 & 0.07 \\
case300 & 0.11  & -- & 0.7   & 0.04 \\
\bottomrule
\end{tabular}
\smallskip

\noindent\footnotesize
Cases from the PGLIB v23 benchmark library~\cite{PGLIB} and
the MATPOWER test suite~\cite{zimmerman2011matpower}.
Gaps for Kocuk et al.~(2016) are taken from the strong SOCP
column.
Gaps for Bynum et al.~(2018) correspond to the RMro (SOCP + McCormick)
method after OBBT convergence.
Gaps for Shao et al.~(2025) correspond to Algorithm~1 at termination.
\end{table}




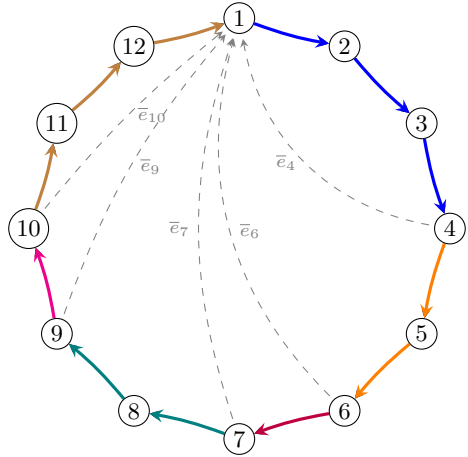
\begin{figure}[ht!]
    \centering
    
\begin{tikzpicture}[scale=2, every node/.style={font=\small},>=stealth]

  \foreach \i in {1,...,12} {
    \node[circle, draw, inner sep=1.5pt] (v\i) at ({90-30*(\i-1)}:1.4) {\i};
  }

  \foreach \i/\j in {1/2,2/3,3/4,4/5,5/6,6/7,7/8,8/9,9/10,10/11,11/12,12/1} {
    \draw[->] (v\i) to[bend right=5] (v\j);
  }


  \draw[line width=1.2pt, blue, ->] (v1) to[bend right=5] (v2) node[midway, above right, font=\scriptsize] {};

  \draw[line width=1.2pt, blue, ->] (v2) to[bend right=5] (v3);
  \draw[line width=1.2pt, blue, ->] (v3) to[bend right=5] (v4);

  \draw[line width=1.2pt, orange, ->] (v4) to[bend right=5] (v5);
  \draw[line width=1.2pt, orange, ->] (v5) to[bend right=5] (v6);

  \draw[line width=1.2pt, purple, ->] (v6) to[bend right=5] (v7) node[midway, right, font=\scriptsize] {};

  \draw[line width=1.2pt, teal, ->] (v7) to[bend right=5] (v8);
  \draw[line width=1.2pt, teal, ->] (v8) to[bend right=5] (v9);

  \draw[line width=1.2pt, magenta, ->] (v9) to[bend right=5] (v10) node[midway, below right, font=\scriptsize] {};

  \draw[line width=1.2pt, brown, ->] (v10) to[bend right=5] (v11);
  \draw[line width=1.2pt, brown, ->] (v11) to[bend right=5] (v12);

  \draw[line width=1.2pt, brown, ->] (v12) to[bend right=5] (v1) node[midway, left, font=\scriptsize] {};

  \tikzset{aux/.style={->, dashed, gray, shorten >=2pt, shorten <=2pt}}

  \draw[aux] (v4) to[bend left=35] node[midway, left, font=\scriptsize] {$\aarc_{4}$} (v1);
  \draw[aux] (v6) to[bend left=30] node[midway, right, font=\scriptsize] {$\aarc_{6}$} (v1);
  \draw[aux] (v7) to[bend left=18] node[midway, left, font=\scriptsize] {$\aarc_{7}$} (v1);
  \draw[aux] (v9) to[bend left=10] node[midway, right, font=\scriptsize] {$\aarc_{9}$} (v1);
  \draw[aux] (v10) to[bend left=6] node[midway, right, font=\scriptsize] {$\aarc_{10}$} (v1);

\end{tikzpicture}

    \caption{Decomposition of a cycle into 3- and 4- cycles.}
    \label{fig:decomposition}
\end{figure}

\subsection{Handling multilinear terms}

\nc{\multilin}{\Phi}
\nc{\monomindex}{t}
\nc{\monomfamily}{T}
\nc{\monomcoef}{a_t}
\nc{\monomset}{J_t}
\nc{\monomsett}{J_{t'}}
\nc{\xv}{x}
\nc{\phigraph}{\Gamma_{\multilin}}
\nc{\convgraph}{\ConvexHull{\Gamma_{\multilin}}}
The cycle constraint \eqref{eq: cycle constraint} is a multilinear constraint, meaning that in each monomial of the sum, every variable appears at most once, and all monomials share the same degree. This class of constraints has been studied through various convex relaxations of equivalent formulations.


Consider a general multilinear function \(\multilin : [\lowb, \upb] \to \bR\), defined as
\begin{equation*}
    \multilin(x) = \sum_{\monomindex \in \monomfamily} \monomcoef \prod_{j \in \monomset} x_j,
\end{equation*}
where \(\lowb, \upb \in \bR^N\) and \([\lowb, \upb] \coloneqq \{\xv \in \bR^N \mid \lowb \leq \xv \leq \upb\}\). 

In~\cite{cycle_Constr, guo2025tightening}, whenever a multilinear function \(\multilin\) appears in a constraint of the considered ACOPF formulation with cycle constraints, it is relaxed by introducing an auxiliary variable \(z\) in place of \(\multilin(x)\) and imposing that \((x,z)\) is contained in the \emph{multilinear polytope} \(\convgraph\), defined as the convex envelope of the set
\begin{equation*}
    \phigraph \coloneqq \{(x,z) \in [\lowb, \upb] \times \bR \mid z = \multilin(x)\}.
\end{equation*}
This approach leads to a linear relaxation since, in~\cite{Rikun1997}, Rikun proved that the convex hull of multilinear functions over rectangular domains is polyhedral. Furthermore, they introduced a class of tight linear inequalities, referred to as associated affine cuts. Nonetheless, no general conditions are known for when these cuts are facet-defining, and the full set of such cuts remains unknown for general multilinear functions. However, a dual formulation of the multilinear polytope can always be computed~\cite{Bao2009}, even though the number of variables required scales exponentially with \(N\).  Meyer~\cite{Meyer2005} showed that the multilinear polytope is uniquely defined by a specific triangulation of the domain \([\lowb, \upb]\), and that the number of facet-defining cuts is lower bounded by the cardinality of a minimal triangulation of this hyperrectangle, which also grows exponentially with \(N\). For this reason, rather than computing the convex envelope of \(\multilin\) directly, convex relaxations are often applied to individual monomials \(\prod_{j \in \monomset} x_j\). 
\nc{\zerov}{\mathbf{0}}
\nc{\onev}{\mathbf{1}}
Theoretical results on the facial structure of the bilinear polytope in the binary case, where \(\lowb = \zerov\) and \(\upb = \onev\), have attracted considerable attention in the literature due to their computational advantages for MINLP. In particular, foundational results by Padberg~\cite{Padberg1989} on the bilinear case have significantly improved the performance of mixed-integer quadratic constrained programming (MIQCP) solvers~\cite{Bao2009, Bonami2018}.
While the multilinear polytope for higher-degree multilinear functions has been extensively studied in the series of works by Del Pia and Khajavirad~\cite{delpia1, delpia2, delpia3, delpia4}, where compact facet-defining formulations are derived for several classes of functions, the non-binary case remains considerably more challenging to address: McCormick~\cite{McCormick1976} and Al-Khayyal and Falk~\cite{AlKhayyal1983} derived the convex envelope of individual bilinear terms over rectangular domains. Meyer and Floudas~\cite{trilinear_convex_envelope} extended these results to trilinear monomials, providing explicit envelopes. However, these constructions are highly intricate, involving 15 separate cases based on variable bounds and requiring nontrivial floating-point operations. For higher-order monomials, approximations are often constructed by recursively applying the bilinear and trilinear envelopes~\cite{AlKhayyal1983, Maranas1995, quadrilinear_approximation}. These approximations yield the exact convex envelope only under restrictive assumptions, such as when \(\lowb = -\upb\), as shown by Leudke et al.~\cite{leudke2012}, who also provide tight bounds on the relaxation gap between the recursive McCormick construction and the true convex hull.

Finally, Costa and Liberti~\cite{DualIsBetterThanPrimal} argue that dual representations of monomial envelopes are more compact and scalable to arbitrary \(N\) compared to their primal counterparts, making them more suitable for practical implementation. To reduce model complexity, these dual formulations can be applied iteratively to generate facet-defining cuts that separate a given point from the convex envelope of the multilinear function~\cite{Bao2009}, preserving the structure of the original problem instead of enforcing a full monolithic formulation. This approach is used, for example, in~\cite{guo2025tightening} to strengthen quadratic convex relaxations for the AC optimal transmission switching problem.

Since the number of variables introduced by the dual formulation grows exponentially with \(N\), it is often computationally advantageous to approximate the convex hull of a multilinear function by computing the convex hull of each monomial separately. This naturally raises the following question: under what conditions does the sum of the convex hulls of the individual monomials coincide with the convex hull of the full multilinear function? Rikun~\cite{Rikun1997} showed that this holds whenever \(\monomset \cap \monomsett = \{\bus\}\) for some fixed \(\bus\), or when the variable sets of all monomials are pairwise disjoint. Meyer~\cite{Meyer2005} extended these results, showing that the equivalence still holds if the triangulations inducing the convex envelopes of the individual monomials project to the same triangulation in the shared variable space. Leudke et al.~\cite{leudke2012} further demonstrated that the sum of monomial convex hulls equals the convex hull of \(\multilin\) whenever \(\lowb \geq 0\).

Finally, we address the following issue: can tighter convex relaxations of multilinear constraints be constructed, at least in principle? For clarity of exposition, we focus on the more specific setting of multilinear equality constraints of the form \(\multilin(x)=0\). Most existing approaches model such constraints by first constructing a convex relaxation of the graph of \(\multilin\), and then enforcing the equality by intersecting this relaxation with the hyperplane \(\{z=0\}\). This modeling strategy is natural, as \(\convgraph\) represents the tightest possible convex relaxation of the graph of a multilinear function. 

However, the convex set of primary interest in this context is not the graph itself, but rather the convex hull of the solution set of the equality constraint \(\multilin(x)=0\). These two constructions need not coincide in general. In particular, intersecting a convex relaxation of the graph with the hyperplane \(\{z=0\}\) may yield a relaxation that is weaker than the convex hull of the feasible set defined directly by the equality constraint, since: \(\{(x,0) \in [\lowb, \upb] \times \bR \mid \multilin(x)=0\}
\;\subset\;
\convgraph \cap \{z=0\}\).
To the best of our knowledge, this distinction has not been explicitly addressed in the literature. Implicitly, it is often assumed that no tighter relaxation can be obtained, since \(\convgraph\) is already the convex hull of the graph of \(\multilin\). Nevertheless, when the objective is to model the feasible region induced by the constraint \(\multilin(x)=0\), a tighter convex relaxation may in fact be available. This is the reason why the various linearizations proposed in~\cite{guo2025tightening}, obtained from equivalent formulations of the ACOPF with cycle constraints, are not equivalent after relaxation, and why imposing that the solution lies in their intersection yields tighter bounds. This motivates a future direction for ACOPF: the study of the convex relaxation of the solution space of the cycle constraints.



\subsection{Convex hull through the dual}
In this subsection we clarify the notion of dual representation. Consider the multilinear monomial:
\begin{equation*}
    \monom(x) \coloneqq \prod_{i=1}^kx_i
\end{equation*}
We let \(\hyperC\coloneqq \prod_{i=1}^k [a_i,b_i]\) be the domain of \(w\), and \(V(\hyperC)\) its set of vertices (note that every vertex \(x \in V(P)\) has as its i-th coordinate, either \(a_i\) or \(b_i\)). Let \(\Gamma_{\monom} = \{(x,w) \in \hyperC \times \bR \mid \monom(x) = w,\; x \in \hyperC \}\) be the graph of \(\monom\) over \(\hyperC\). It is known~\cite{Rikun1997}, that its convex hull \(\ConvexHull{\Gamma_{\monom}}\) is the convexification of the set \(\{(x,\monom(x)) \mid x \in V(\hyperC)\}\). That is,
\begin{equation*}
\ConvexHull{\Gamma_{\monom}} =
\left\{
(x,w) \;\middle|\;
\begin{aligned}
& \exists \lambda_v \ge 0 \quad && \forall v \in V(\hyperC), \\
& \sum_{v \in V(\hyperC)} \lambda_v = 1, \\
& (x,w) = \sum_{v \in V(\hyperC)} \lambda_v \bigl(v, \monom(v)\bigr)
\end{aligned}
\right\}.
\end{equation*}
This means that introducing the variables \(\lambda_v\) for all \(v \in \hyperC\), \(x_i\) for \(i=1,\ldots,\bus\) and \(w\) as done in~\cite{DualIsBetterThanPrimal}, we can rewrite the convex hull of \(\monom\) as a linear formulation as follows:
\begin{subequations}
    \begin{align}\label{eq: dual formulation}
    w &= \sum_{v \in V(\hyperC)} \lambda_v \monom(v) \\
    x_i &= \sum_{v \in V(\hyperC)} \lambda_v v_i \\
    \sum_{v \in V(\hyperC)} \lambda_v &= 1 \\
    \lambda &\geq 0 \; \forall v \in V(\hyperC).
\end{align}
\end{subequations}
    The number of introduced variables is \(\mathcal{O}(2^k)\). Now consider a general multilinear function \(\multilin\) over \(\hyperC\), that can be written as \[\multilin = \sum_{i=1}^n \monom_i,\] with \(\monom_i: \hyperC_0 \times \hyperC_i \to \bR\) and \(\hyperC = \prod _{i=0}^n\hyperC_i\). \(\hyperC_0\) corresponds to the variables shared by all monomials, while \(\hyperC_i\) are only present in \(\omega_i\). Similarly to the monomial case, the convex hull of \(\Gamma_{\multilin}\) can be modeled by:  
    \begin{equation}\label{eq: multilinear dual formulation}
        \ConvexHull{\Gamma_{\multilin}}=\ConvexHull{\{(x,w) \in \bR^{k+1} \mid \multilin(x) = w,\; x \in V(\hyperC) \}}.
    \end{equation}
    Equation \eqref{eq: multilinear dual formulation} is used in \cite{guo2025tightening} to implement bilinear constraints. However, rather than embedding the dual convex envelope of the cycle constraints directly into the model, their approach operates iteratively: given a candidate solution, they verify whether it lies within the dual representation of the convex envelope of each cycle constraint. If it does not, a corresponding Benders cut is generated and added to the main problem. The problem is then re-solved to obtain an updated candidate solution. The problem is repeatedly solved until a candidate solution is found which lies within the dual representation of the convex envelope of each cycle constraint.
    
 A drawback of the dual formulation~\eqref{eq: multilinear dual formulation} is that it introduces \(\mathcal{O}\!\left(2^{\sum_{i=0}^n n_i}\right)\) variables. Indeed, Proposition~\ref{prop: multilin convex hull}\footnote{In~\cite{Meyer2005}, \(P_0\) is assumed to be a simplex; here we provide a simple but useful generalization in which \(P_0\) is an arbitrary hypercube.} first introduced in~\cite{Meyer2005}, provides conditions under which modeling each monomial separately using~\eqref{eq: dual formulation} yields the same relaxation as applying the dual formulation to the entire multilinear function via~\eqref{eq: multilinear dual formulation}. Under these conditions, the number of introduced variables is reduced to \(\mathcal{O}\!\left(\sum_{i=0}^n 2^{\,n_0 + n_i}\right)\).

\begin{proposition}\label{prop: multilin convex hull}
Consider a multilinear function \(f\) over \(\hyperC\), that can be written as \[f = \sum_{i=1}^n \monom_i,\] with \(\monom_i: \hyperC_0 \times \hyperC_i \to \bR\) linear over \(\hyperC_0\) and \(\hyperC = \prod _{i=0}^n\hyperC_i\). Then the convex hull of \(f\) over \(\hyperC\) is given by:
\[
\ConvexHull{\Gamma_{f}} =\{(x,w) \in \bR^{\sum_{i=0}^n n_i + 1} \mid w = \sum_{i=1}^n w_i,\; (x_0,x_i,w_i) \in \ConvexHull{\Gamma_{\monom_i}}\}.
\]
\end{proposition}

Proposition \ref{prop: multilin convex hull} follows  directly from the following Theorem in~\cite{Rikun1997}:
\begin{theorem}\label{thm: Rikun}
Let \(\hyperC\) be a Cartesian product of polytopes, \(\hyperC = \prod_{i=0}^n \hyperC_i\), \(\hyperC_i \subset \bR^{n_i}\), and let \(\monom_i(x_0,x_i)\) be a continuous function defined on \(\hyperC_0\times \hyperC_i,\;i=1,\ldots,n\). If each \(\monom_i(x_0,x_i)\) is a concave function of \(x_0\) when \(x_i\) is fixed and \(P_0\) is a simplex, then
\begin{equation*}
    \ConvexHull{\sum_{i=1}^n\monom_i}(x) = \sum_{i=1}^n \ConvexHull{\monom_i}(x_0,x_i).
\end{equation*}
\end{theorem}

In particular, if \(\omega_i\) are linear functions with respect to \(x_0\), when \(x_i\) is fixed, we have:
\begin{align*}
        \ConvexHull{\sum_{i=1}^n\monom_i}(x) = \sum_{i=1}^n \ConvexHull{\monom_i}(x_0,x_i) \\
        \ConcaveHull{\sum_{i=1}^n\monom_i}(x) = \sum_{i=1}^n \ConcaveHull{\monom_i}(x_0,x_i)
\end{align*}
We can now prove Proposition~\ref{prop: multilin convex hull}.
\begin{proof}
     First assume that \(P_0\) is a simplex, that is the convex hull of \(n_0+1\) affinely independent points. Then
     \((x,w) \in \ConcaveHull{\Gamma_{\multilin}} \) if and only if 
     \[ \sum_{i=1}^n \ConvexHull{\monom_i}(x_0,x_i) = \ConvexHull{f}(x) \leq w \leq \ConvexHull{f}(x) =  \sum_{i=1}^n \ConcaveHull{\monom_i}(x_0,x_i). \] Since \(\ConvexHull{\monom_i}(x_0,x_i) \leq \ConcaveHull{\monom_i}(x_0,x_i)\) and \((w_1,\ldots,w_n) \mapsto \sum w_i\) is continuous, by the intermediate value theorem there exists \(w_i \in \bR \) such that \(\ConvexHull{\monom_i}(x_0,x_i) \leq w_i \leq \ConcaveHull{\monom_i}(x_0,x_i)\) and \(\sum_i w_i = w\). Thus \((x_0,x_i,w_i) \in \ConvexHull{w_i}\) and \((x,w) \in \{(x,w) \in \bR^{\sum_{i=0}^n n_i + 1} \mid w = \sum_{i=1}^n w_i,\; (x_0,x_i,w_i) \in \ConvexHull{\Gamma_{\monom_i}}\}\). Conversely, if \((x_0,x_i,w_i) \in \ConvexHull{\Gamma_{\monom_i}}\) then clearly \((x, w) \in \ConvexHull{\Gamma_{\multilin}}\) with \(w \coloneqq \sum_i w_i\).

  In particular, for \(n_0 = 0\), we obtain the trivial case where the sum involves monomials with no variables in common.  
For \(n_0 = 1\), since the corresponding interval is a simplex, the proposition holds for generic intervals—this corresponds to the case in which all monomials share exactly one variable.  

When \(P_0\) is a generic hypercube, note that \(\omega_j\) is linear with respect to the variables in \(P_0\) when the variables in \(P_i\) (for \(i > 0\)) are fixed. Hence, exactly one variable from \(P_0\) appears in each \(\omega_j\).  
We can therefore proceed in two steps: first, group all monomials sharing the same variable together (the case \(n_0 = 0\)), and then apply the result to each such group (the case \(n_0 = 1\)).
\end{proof}
Thus, \(\ConvexHull{\Gamma_{f}}\) can be represented by:

\begin{subequations}
   \begin{align}
    w &= \sum_{i=1}^n w_i \\
    w_i &= \sum_{v \in V(\hyperC_0)\times V(\hyperC_i)} \lambda_v \monom_i(v_0,v_i) \quad \forall i=1,\ldots,n\\
    (x_0,x_i) &= \sum_{v \in  V(\hyperC_0)\times V(\hyperC_i)} \lambda_v (v_0,v_i)  \quad \forall i=1,\ldots,n\\
    \sum_{v \in  V(\hyperC_0)\times V(\hyperC_i)} \lambda_v &= 1  \quad \forall i=1,\ldots,n \\
    \lambda &\geq 0 \; \forall v \in V(\hyperC)
\end{align} 
\end{subequations}
where for every \(v \in V(\hyperC_0)\times V(\hyperC_i)\), \(v_j \coloneqq \pi_{V(\hyperC_j)} \in \bR^{n_j}\), with \(\pi\) the linear projection.

In the context of ACOPF, it is of particular interest to understand how these results apply to the various formulations of the cycle constraints. In particular, one can verify that constraints~\eqref{eq: 3 cycle constraint}, \eqref{eq: 3 cycle constr var}, \eqref{eq: 4 cycle constraint}, and~\eqref{eq: 4 cycle constr var} all satisfy the conditions of Proposition~\ref{prop: multilin convex hull}. As a consequence, modeling the convex hull of each of these individual multilinear functions is equivalent to modeling the convex hull of their constituent monomials separately.

However, Proposition~\ref{prop: multilin convex hull} does not address how these two relaxations compare when multiple multilinear constraints are imposed simultaneously. In the next section we explore this direction for the case of bilinear constraints, showing that the full dual formulation of a collection of bilinear constraints is equivalent to the McCormick formulation whenever the associated interaction graph is a tree.
\subsection{Bilinear convex cut}
\nc{\vars}{V}
\nc{\edges}{E}
\nc{\graph}{G}
\nc{\MC}{MC}

Bilinear functions can be represented by a graph in which each variable corresponds to a node and each binomial term defines an edge connecting the two associated variables. In this section we extend known results on the exactness of the convex hull of a single bilinear function to the convex hull of multiple bilinear functions. Consider the variable set \(\vars\) and the edge set \(\edges\), where each \((i,j)\in E\) indicates that the binomial term \(x_i x_j\) appears in one of the bilinear functions, with \(x_i \in [x_i^0,x_i^1]\) and denote \(\graph = (\vars,\edges)\). The convex hull of the multigraph \(\Gamma = \{(x_i,y_e)_{i \in \vars, \arc \in \edges} \mid x_i\in [x_i^0,x_i^1], y_{(i,j)}=x_ix_j\) \} admits the following extremal-point representation:
\begin{align*}
     \ConvexHull{\Gamma(\graph)}  = \{ & (x_i,y_e)_{i \in \vars, \arc \in \edges}\} \mid x_i\in [x_i^0,x_i^1], x_i = \\ & = \sum_{w \in \{0,1\}^V}\lambda_w x^{w_i}_i,\, y_{(h,k)} = \\
     & = \sum_{w \in \{0,1\}^V}\lambda_w x^{w_h}_hx^{w_{\bus}}_{\bus}, \text{where \(\lambda_{w}\) are the} \\ & \text{ coefficients of a convex linear combination} \}.
\end{align*}
   
As this representation entails  the use of \(2^{|\vars|}\) additional variables, we are interested in when \(\ConvexHull{\Gamma}\) can be obtained by considering the  McCormick relaxation for each binomial instead, that is the set:
\begin{align*}
    \MC(\graph) = \{ & (x_i,y_e)_{i \in \vars, \arc \in \edges}\} \mid x_i\in [x_i^0,x_i^1],\;  x_j\in [x_j^0,x_j^1], \\ 
    & x_i = \sum_{w \in \{0,1\}^2}\lambda^{i,j}_w x^{w_0}_i,\,x_j = \\
    & = \sum_{w \in \{0,1\}^2}\lambda^{i,j}_w x^{w_1}_j,\, y_{(i,j)} = \\
    & = \sum_{w \in \{0,1\}^V}\lambda_w x^{w_h}_hx^{w_{\bus}}_{\bus}, \text{ where } \sum_{w \in \{0,1\}^2}\lambda^{(i,j)}_w=1\}.
\end{align*}

\begin{proposition}\label{prop: binomes convex hull}
    If \(\graph\) is a tree, then 
    \begin{equation*}
        \ConvexHull{\Gamma} = \MC(\graph).
    \end{equation*}
\end{proposition}
\begin{proof}
   If \(|E|=1\), the formulations coincide so the statement is trivially true. Assume the statement holds for all trees with \(m-1\) edges, and let us prove it for a tree with \(m\) edges.  
Let \((x_i, x_e)_{i \in \vars,\, e \in \edges} \in \MC(\graph)\).  
Choose a leaf \(k \in \vars\), and let \((h,k) \in \edges\) be its unique incident edge.  
Set \(\edges' \coloneqq \edges \setminus \{(h,k)\}\) and \(\vars' \coloneqq \vars \setminus \{\bus\}\), and define \(\graph' \coloneqq (\vars', \edges')\).

By construction, \((x_i, y_e)_{i \in \vars',\, e \in \edges'} \in \MC(\graph')\).  
By the induction hypothesis, \((x_i, y_e)_{i \in \vars',\, e \in \edges'} \in \ConvexHull{\Gamma'}\).  
Hence there exist coefficients \(\lambda'_{w'}\) for \(w' \in \{0,1\}^{V'}\) such that  
\[
    x_i = \sum_{w' \in \{0,1\}^{V'}} \lambda'_{w'} x_i^{w'_i},
    \qquad
    y_{(i,j)} = \sum_{w' \in \{0,1\}^{V'}} \lambda'_{w'} x_i^{w'_i} x_j^{w'_j}
    \quad \text{for all } (i,j) \in \edges'.
\]

Since \((x_h, x_{\bus}, y_{(h,k)}) \in \MC(\graph)\), there exist convex combination weights \(\lambda^{h,k}_{w}\) for \(w \in \{0,1\}^2\) such that  
\[
    (x_h, x_{\bus}, y_{(h,k)}) 
    = \sum_{w \in \{0,1\}^2} 
    \lambda^{h,k}_{w}\,\bigl(x^{w_h}_h,\; x^{w_{\bus}}_{\bus},\; x^{w_h}_h x^{w_{\bus}}_{\bus} \bigr).
\]

We now extend coefficients \(\lambda'_{w'}\) (indexed by \(w' \in \{0,1\}^{V'}\)) to coefficients \(\lambda_w\) indexed by \(w \in \{0,1\}^{V}\), so that  
\[
    \lambda'_{w'} 
    = \lambda_{(w',0)} + \lambda_{(w',1)}
    \qquad \text{for all } w' \in \{0,1\}^{V'}.
\]

For \(i \in \vars'\) and \(z \in \{0,1\}\), define  
\[
    \mu_i^z \coloneqq \sum_{\substack{w' \in \{0,1\}^{V'} \\ w'_i = z}} \lambda'_{w'}.
\]
Note that \(\mu_h^z = \lambda^{h,k}_{(z,\,0)} + \lambda^{h,k}_{(z,\,1)}\).
We then set
\[
\lambda_{(w',z)} 
=  
\begin{cases}
\lambda'_{w'}\dfrac{\lambda^{h,k}_{(z',z)}}{\mu_h^{z'}} & \text{if } \mu_h^{z'} > 0,\\
0 & \text{otherwise}.
\end{cases}
\]

For the case $\mu_h^{w'_h} \neq 0$, we show that these coefficients form a valid convex combination of the extreme points of \(\ConvexHull{\Gamma}\), yielding \((x_i, y_{h,k})_{i \in \vars,\,(h,k)\in\edges}\). For all \(i \neq \bus\) we have:

\begin{align*}
    \sum_{w \in \{0,1\}^V} \lambda_w\, x_i^{w_i}
    &= \sum_{\substack{w \in \{0,1\}^V \\ w_{\bus} = 0}} 
        \lambda_{(w',0)}\, x_i^{w'_i}
     + \sum_{\substack{w \in \{0,1\}^V \\ w_{\bus} = 1}} 
        \lambda_{(w',1)}\, x_i^{w'_i} \\
    &= \sum_{w' \in \{0,1\}^{V'}} 
        \lambda'_{w'} \left(
            \frac{\lambda^{h,k}_{(w'_h,0)}}{\mu_h^{\,w'_h}}
            + 
            \frac{\lambda^{h,k}_{(w'_h,1)}}{\mu_h^{\,w'_h}}
        \right) x_i^{w'_i} \\
    &= \sum_{w' \in \{0,1\}^{V'}} 
        \lambda'_{w'} \, x_i^{w'_i} = x_i,
\end{align*}
where the last equality follows from  
\[
\frac{\lambda^{h,k}_{(w'_h,0)}}{\mu_h^{\,w'_h}}
\;+\;
\frac{\lambda^{h,k}_{(w'_h,1)}}{\mu_h^{\,w'_h}}
= 1.
\]

If \(i = \bus\), then
\begin{align*}
    \sum_{w \in \{0,1\}^V} \lambda_w x_{\bus}^{w_{\bus}}
    &= 
    \sum_{\substack{w \in \{0,1\}^V \\ w_{\bus} = 0}} 
        \lambda_{(w',0)} \, x_{\bus}^{0}
    +
    \sum_{\substack{w \in \{0,1\}^V \\ w_{\bus} = 1}} 
        \lambda_{(w',1)} \, x_{\bus}^{1} 
        \\
    &=
    \sum_{w' \in \{0,1\}^{V'}}
        \lambda'_{w'} 
        \frac{\lambda^{h,k}_{(w'_h,0)}}{\mu_h^{\,w'_h}}
        \, x_{\bus}^{0}
    +
    \sum_{w' \in \{0,1\}^{V'}} 
        \lambda'_{w'}
        \frac{\lambda^{h,k}_{(w'_h,1)}}{\mu_h^{\,w'_h}}
        \, x_{\bus}^{1}.
\end{align*}

The first sum expands as
\[
\sum_{w' \in \{0,1\}^{V'}}
    \lambda'_{w'} 
    \frac{\lambda^{h,k}_{(w'_h,0)}}{\mu_h^{\,w'_h}} 
    x_{\bus}^{0}
=
\sum_{\substack{w' \in \{0,1\}^{V'} \\ w'_h = 0}}
    \lambda'_{w'} 
    \frac{\lambda^{h,k}_{(0,0)}}{\mu_h^{0}}
    x_{\bus}^{0}
+
\sum_{\substack{w' \in \{0,1\}^{V'} \\ w'_h = 1}}
    \lambda'_{w'}
    \frac{\lambda^{h,k}_{(1,0)}}{\mu_h^{1}}
    x_{\bus}^{0}.
\]

Using the definitions
\[
\mu_h^0 = \sum_{\substack{w' : w'_h = 0}} \lambda'_{w'},
\qquad
\mu_h^1 = \sum_{\substack{w' : w'_h = 1}} \lambda'_{w'},
\]
we obtain
\[
\sum_{w' \in \{0,1\}^{V'}}
    \lambda'_{w'} 
    \frac{\lambda^{h,k}_{(w'_h,0)}}{\mu_h^{\,w'_h}} 
    x_{\bus}^{0}
=
\lambda^{h,k}_{(0,0)} + \lambda^{h,k}_{(1,0)}.
\]

Similarly,
\[
\sum_{w' \in \{0,1\}^{V'}} 
    \lambda'_{w'}
    \frac{\lambda^{h,k}_{(w'_h,1)}}{\mu_h^{\,w'_h}}
    x_{\bus}^{1}
=
\lambda^{h,k}_{(0,1)} + \lambda^{h,k}_{(1,1)}.
\]

Hence,
\[
\sum_{w \in \{0,1\}^V} \lambda_w x_{\bus}^{w_{\bus}}
    =
    x^0_{\bus} \cdot (\lambda^{h,k}_{(0,0)} + \lambda^{h,k}_{(1,0)})
    +
    x^1_{\bus} \cdot (\lambda^{h,k}_{(0,1)} + \lambda^{h,k}_{(1,1)})
    = x_{\bus},
\]

A similar verification applies to the edge variables \(y_{i,j}\) for \((i,j) \in \edges'\) and to edge \((h,k)\), as well as in the case $\mu_h^{w'_h} = 0$. This follows from the fact that $0 = \mu_h^{w'_h} = \lambda^{h,k}_{(w'_h,\,0)} + \lambda^{h,k}_{(w'_h,\,1)}$ implies
$\lambda^{h,k}_{(w'_h,0)} = \lambda^{h,k}_{(w'_h,1)} = 0$, and consequently 
$\lambda'_{w'} = \lambda^{h,k}_{(w'_h,0)} +\lambda^{h,k}_{(w'_h,1)} = 0$ for all such $w'$.
\end{proof}

Note that the graph associated with the bilinear terms of the cycle constraint formulation~\eqref{eq: 3 cycle constr var} is indeed a tree, and therefore Proposition~\ref{prop: binomes convex hull} applies. In contrast, this proposition cannot be applied to the formulation~\eqref{eq: 4 cycle constr var}, for which the associated graph is not a tree.



\section{Conclusions}
This paper presented a structured review of strengthening techniques for the Jabr second-order cone relaxation of the ACOPF and developed a unified perspective grounded in multilinear convexification and graph structure. By interpreting cycle consistency conditions as multilinear equalities and analyzing their interaction graphs, we connected cycle-based cuts, McCormick relaxations, and dual extended formulations within a common geometric framework.

Our analysis clarifies that strengthening the Jabr relaxation can be understood as convexifying structured multilinear constraints induced by network cycles. In particular, we identify structural conditions under which primal and dual relaxations coincide and highlight the distinction between convexifying interaction graphs and convexifying feasible sets.

This viewpoint suggests several directions for further research, including tighter characterizations of cycle-based convex hulls, scalable extended formulations for dense subgraphs, and integration of graph-structured convexification techniques within global optimization frameworks. We hope that the unified framework presented here will facilitate the systematic design of stronger and more interpretable relaxations for ACOPF and related network optimization problems.

\section*{Acknowledgments}

The authors thank Nicol\`o Gusmeroli and Pietro Belotti for reading an earlier version of this manuscript and for their valuable suggestions.
\par 
\medskip
Gabor Riccardi and Stefano Gualandi acknowledge financial support under the National Recovery and Resilience Plan (NRRP), Mission 4, Component 2, Investment 1.1, Call for tender No. 1409 published on 14.9.2022 by the Italian Ministry of University and Research (MUR), funded by the European Union -- NextGenerationEU-- Project Title HEXAGON: Highly--specialized EXact Algorithms for Grid Operations at the National level -- CUP F53D23010010001 - Grant Assignment Decree No. 1379 adopted on 01/09/2023 by the Italian Ministry of Ministry of University and Research (MUR).

\bibliographystyle{plain} 
\bibliography{biblio}

\end{document}